\documentclass[12pt]{amsart}
\usepackage{verbatim}
\usepackage{amssymb, amsmath}
\usepackage{graphicx}
\usepackage{caption}
\usepackage{subcaption}
 \usepackage{enumitem}
\usepackage{color}
\usepackage{url}

 \usepackage[colorlinks,linkcolor=blue,anchorcolor=blue,citecolor=blue,backref=page]{hyperref}

\usepackage{hyperref}

\begin{document}
\newtheorem{thm}{Theorem}[section]
\newtheorem*{thm*}{Theorem}
\newtheorem{lem}[thm]{Lemma}
\newtheorem{prop}[thm]{Proposition}
\newtheorem{cor}[thm]{Corollary}
\newtheorem{conj}{Conjecture}
\newtheorem{proj}[thm]{Project}
\newtheorem*{question}{Question}
\newtheorem{rem}{Remark}[section]

\theoremstyle{definition}
\newtheorem*{defn}{Definition}
\newtheorem*{remark}{Remark}
\newtheorem{exercise}{Exercise}
\newtheorem*{exercise*}{Exercise}

\numberwithin{equation}{section}

\newcommand{\rad}{\operatorname{rad}}

\newcommand{\Z}{{\mathbb Z}} 
\newcommand{\Q}{{\mathbb Q}}
\newcommand{\R}{{\mathbb R}}
\newcommand{\C}{{\mathbb C}}
\newcommand{\N}{{\mathbb N}}
\newcommand{\FF}{{\mathbb F}}
\newcommand{\fq}{\mathbb{F}_q}
\newcommand{\rmk}[1]{\footnote{{\bf Comment:} #1}}

\renewcommand{\mod}{\;\operatorname{mod}}
\newcommand{\ord}{\operatorname{ord}}
\newcommand{\TT}{\mathbb{T}}
\renewcommand{\i}{{\mathrm{i}}}
\renewcommand{\d}{{\mathrm{d}}}
\renewcommand{\^}{\widehat}
\newcommand{\HH}{\mathbb H}
\newcommand{\Vol}{\operatorname{vol}}
\newcommand{\area}{\operatorname{area}}
\newcommand{\tr}{\operatorname{tr}}
\newcommand{\norm}{\mathcal N} 
\newcommand{\intinf}{\int_{-\infty}^\infty}
\newcommand{\ave}[1]{\left\langle#1\right\rangle} 
\newcommand{\Var}{\operatorname{Var}}
\newcommand{\Prob}{\operatorname{Prob}}
\newcommand{\sym}{\operatorname{Sym}}
\newcommand{\disc}{\operatorname{disc}}
\newcommand{\CA}{{\mathcal C}_A}
\newcommand{\cond}{\operatorname{cond}} 
\newcommand{\lcm}{\operatorname{lcm}}
\newcommand{\Kl}{\operatorname{Kl}} 
\newcommand{\leg}[2]{\left( \frac{#1}{#2} \right)}  
\newcommand{\Li}{\operatorname{Li}}

\newcommand{\sumstar}{\sideset \and^{*} \to \sum}

\newcommand{\LL}{\mathcal L} 
\newcommand{\sumf}{\sum^\flat}
\newcommand{\Hgev}{\mathcal H_{2g+2,q}}
\newcommand{\USp}{\operatorname{USp}}
\newcommand{\conv}{*}
\newcommand{\dist} {\operatorname{dist}}
\newcommand{\CF}{c_0} 
\newcommand{\kerp}{\mathcal K}

\newcommand{\Cov}{\operatorname{cov}}
\newcommand{\Sym}{\operatorname{Sym}}

\newcommand{\Ht}{\operatorname{Ht}}

\newcommand{\E}{\operatorname{\mathbb E}} 
\newcommand{\sign}{\operatorname{sign}} 
\newcommand{\meas}{\operatorname{meas}} 
\newcommand{\length}{\operatorname{length}} 

\newcommand{\divid}{d} 

\newcommand{\GL}{\operatorname{GL}}
\newcommand{\SL}{\operatorname{SL}}
\newcommand{\re}{\operatorname{Re}}
\newcommand{\im}{\operatorname{Im}}
\newcommand{\res}{\operatorname{Res}}
 \newcommand{\eigen}{\Lambda} 
\newcommand{\tens}{\mathbf t} 
\newcommand{\diam}{\operatorname{diam}}
\newcommand{\fixme}[1]{\footnote{Fixme: #1}}
 \newcommand{\EWp}{\mathbb E^{\rm WP}_g} 
\newcommand{\orb}{\operatorname{Orb}}
\newcommand{\supp}{\operatorname{Supp}}
\newcommand{\mmfactor }{\textcolor{red}{c_{\rm Mir}}}
\newcommand{\Mg}{\mathcal M_g} 
\newcommand{\MCG}{\operatorname{Mod}} 
\newcommand{\Diff}{\operatorname{Diff}} 
\newcommand{\If}{I_f(L,\tau)}

\newcommand{\GOE}{\operatorname{GOE}}
\newcommand{\GUE}{\operatorname{GUE}}
\newcommand{\GSE}{\operatorname{GSE}}
\newcommand{\rest}{\operatorname{rest}} 
\newcommand{\diag}{\operatorname{DIAG}} %
\newcommand{\off}{\operatorname{OFF}} %

\title[Closed geodesics in short intervals]{Closed geodesics in short intervals for random hyperbolic surfaces}
\author{Ze\'ev Rudnick}
\address{School of Mathematical Sciences, Tel Aviv University, Tel Aviv 69978, Israel} 
\email{rudnick@tauex.tau.ac.il}
\date{\today}
 
\begin{abstract}
We study the distribution of closed geodesics in short intervals on random hyperbolic surfaces of large genus, and compare it with the classical problem of primes in short intervals. Viewing the surface $M$ as a random point in moduli space equipped with the Weil--Petersson measure, we investigate the random variable $\Psi_M(x;H)$ counting closed geodesics with norms in the interval $[X, X+H]$, weighted by primitive length, where $H=o(X)$. This is analogous to the Chebyshev function in prime number theory. 

Our main result establishes that  in the large genus limit,  
\[
\lim_{g\to \infty}\mathrm{Var}(\Psi_M(X;H)) \sim 2\,H \log X,
\]
when $X\to \infty$,  $H=o(X)$.

Goldston and Montgomery related the variance for primes in short
intervals to the form factor associated with zeros of the Riemann zeta
function, and conjectured that it is asymptotic to
\[
H\log(X/H).
\]
 We show that for automorphic
$L$-functions of degree $d>1$, the early-time GUE form factor
already follows from the Riemann Hypothesis, thereby recovering
the variance $H\log X$ in the very short interval regime predicted
by Bui, Keating and Smith. 

In the geometric setting, the appearance of $\log X$ reflects the
much higher spectral density of Laplace eigenvalues relative to
zeros of finite-degree $L$-functions, while the additional factor of
$2$ is explained by the expected GOE statistics for the Laplace
spectrum of generic hyperbolic surfaces.
\end{abstract}

\keywords{Moduli space, hyperbolic surface, closed geodesics, primes in short intervals,   Random Matrix Theory, pair correlation function, form factor.}

\maketitle

 \tableofcontents

\section{Introduction: Closed geodesics vs primes in short intervals}

 We investigate the distribution of closed geodesics on hyperbolic surfaces in short intervals, with the aim of comparing their statistics to those of prime numbers. The analogy between primes and geodesics, embodied in the prime geodesic theorem, suggests that questions about primes in short intervals should have geometric counterparts. We approach this problem from a probabilistic perspective, viewing the surface as random in moduli space with respect to the Weil–Petersson measure. Our goal is to understand the typical behavior, for large genus, of the appropriately weighted number of closed geodesics with lengths in short intervals, and to relate it to conjectural predictions from analytic number theory arising from random matrix statistics of zeros of L-functions.

\subsection{Primes} 
Recall that the Prime Number Theorem is equivalent to the assertion that $\psi(x)\sim x$ as $x\to \infty$, where 
\[
\psi(x):=\sum_{n\leq x} \Lambda(n)
\]
and  $\Lambda(n)$ is the von Mangoldt function, equal to $\log p$ if $n=p^k$ is a power of a prime $p$, and equal to zero otherwise.

A key objective is to understand the distribution of primes in short intervals, and to do so one studies
\[
\psi(x;H):=\psi(x+H)-\psi(x)=\sum_{x<n\leq x+H} \Lambda(n)
\]
which helps count the number of primes in short intervals.  The Riemann Hypothesis (RH) guarantees an asymptotic formula
$\psi(X;H)\sim H$ as long as $H>X^{\frac 12 +o(1)}$, $H=o(X)$, see \cite{GuthMaynard} for the currently best unconditional result towards this, and the asymptotic is expected to hold for all intervals if $X^\varepsilon<H<X^{1-\varepsilon}$, $\forall \varepsilon>0$ (but not for shorter intervals \cite{Maier}).  
Selberg   \cite{Selberg} showed that the asymptotic formula holds for {\em almost all}  intervals of length $H=o(X)$ in $[1,X]$,  $H/(\log X)^2\to \infty$, assuming RH    (unconditionally for $H\gg X^{19/77}$; the exponent $19/77$ has been improved, to the current value of $2/15$ \cite{GuthMaynard}).

Goldston and Montgomery
\cite{GM} studied the variance of $\psi(x;H)$ and conjectured 
\begin{conj}\label{GM conjecture}
For every fixed $0<\delta<1$, 
\begin{equation*}
\Var(\psi)(X;H):=\frac 1X\int_1^X \left|\psi(x;H)- H\right|^2 dx \sim H \log \frac XH 
\end{equation*}
holds uniformly for $1\leq H<X^{1-\delta}$.
\end{conj}
They showed this to hold, assuming the Riemann Hypothesis and  the (``strong") pair correlation conjecture for the zeros of the Riemann zeta function.   See \cite{KeatingRudnick, HKR} for a proof in a function field context. 

\subsection{Closed geodesics on hyperbolic surfaces}

Let $M$ be a closed hyperbolic surface of genus $g\geq 2$. 
We define the {\em norm} of a closed geodesic on $M$  as 
$$N(\gamma):=e^{\ell(\gamma)},$$ 
where $\ell(\gamma)=\ell_M(\gamma)$  is the length. 

Let $\Pi_M(x)$ be the number of   primitive,  closed,  oriented   geodesics of  norm is at most $x$, equivalently of length $\ell_M(\gamma)\leq \log x$, and let $\Psi_M(x)$ be the   count of closed  oriented  
 geodesics of norm $N(\gamma)\leq x$, weighted by the length  of the corresponding primitive geodesic:
\[
\Psi_M(x)=\sum_{N(\gamma)\leq x} \Lambda(\gamma)
\] 
where the von Mangoldt function is
$$
\Lambda(\gamma) =\log N(\gamma_0)= \ell(\gamma_0),
$$ 
if $\gamma=\gamma_0^k$ with  $\gamma_0$ primitive, and $k\geq 1$.

The Prime Geodesic Theorem \cite{Selberg1954, Huber} says that\footnote{For counting unoriented geodesics, we multiply by $\tfrac 12$.} 
\[  
\Pi_M(x) \sim  \Li(x) 
\]
where $\Li(x):=\int_2^x\frac1{\log u} du\sim \frac{x}{\log x}$. 
In terms of the Chebyshev function $\Psi_M(x)$, the PGT is equivalent to  $\Psi_M(x)\sim x$. A version with remainder term is 
\[
\Psi_M(x) = x+\sum_{\frac 12 <s_j<1} \frac{x^{s_j}}{s_j} +  O_M(x^{3/4})  
\] 
where the sum is over ``small" eigenvalues $\lambda_j=s_j(1-s_j)\in (0,\frac 14)$ of the Laplace-Beltrami operator on $M$. 
It is expected that the exponent $3/4$ in the remainder term should be  $1/2+\varepsilon$, $\forall \varepsilon>0$.  For improvements in the special case of the modular surface and congruence subgroups, see \cite{Iwaniec, LuoSarnakPIHES, Cai, SoundYoung, AcostaReche}.




 \subsection{Short intervals}
We want to count geodesics in ``short" intervals, so define for $H=o(X)$   
\[
\Psi_M(x;H):=\Psi_M(x+H)-\Psi_M(x)
\]
which counts oriented geodesics of norm $x<N(\gamma)\le x+H$. As long as $H\gg  x^{3/4}$,    
we know that 
\begin{equation}\label{psi in short intervals}
 \Psi_M(x;H)\sim H, \quad x\to \infty.  
\end{equation}
For the special case of the modular surface, Bykovskii \cite{Bykovskii} showed \eqref{psi in short intervals} holds for $H\gg x^{1/2+\varepsilon}$, which is sharp in that case.  

For  compact arithmetic surfaces,  \eqref{psi in short intervals} cannot hold for $H\ll x^{1/2}$; for instance for those studied by Selberg \cite[Chapter 18]{Hejhal STF1}, the norms are of the form $4m^2-2 + O\left( \frac 1{m^2} \right)$ for integers $m\geq 1$,  so that there are intervals of size $\asymp x^{1/2}$ which contains no norms, and  a similar feature holds for all arithmetic surfaces \cite{LuoSarnak}.  There are certain non-arithmetic  surfaces (but ``semi-arithmetic" in a suitable sense) where there is  some $\delta>0$ so that there are intervals of length $H=x^\delta$ which contain no norms \cite{BGGS97, BS04, Belolipetskytal}.  
However,  for  generic hyperbolic surfaces, we might expect \eqref{psi in short intervals} to hold  for   $H= x^\varepsilon $ for all $\varepsilon>0$.


To gain perspective on this issue, we turn to studying the problem when the surface $M$ varies in the moduli space $\Mg$ of all hyperbolic surfaces of given genus $g$, which comes equipped with a natural probability measure, the Weil-Petersson measure. In this context, Wu and Xue \cite{WuXue} showed that with high probability in $\Mg$, there is a uniform version of the remainder term in the Prime Geodesic Theorem: for all $\varepsilon>0$, and all $X=X(g)>2$, with probability tending to one as  $g\to \infty$, we have $|\Pi_M(X)-\Li(X)|\leq c\, g\, X^{3/4+\varepsilon}$ with $c>0$ independent of the surface $M$ and of $g$.

We wish to study the variance of $\Psi_M(X;H)$, viewed as a random variable on $\Mg$, setting
\[
\Var^{WP}_g\left(\Psi_M(X;H)\right) :=\EWp\left( | \Psi_M(X;H)-H |^2 \right)
\]
 where $\EWp$ denotes the expected value with respect to the Weil-Petersson measure.
That the expected value  $\lim_{g\to \infty}\EWp(\Psi_M(X;H))$ is indeed asymptotically $H$ for $X\to \infty$ while $H=o(X)$, is shown in \S~\ref{sec:expectation}.


Our main result is: 
 \begin{thm}\label{main thm}
For  $X\to \infty$, $H=o(X)$ we have 
\[
 \lim_{g\to \infty} \Var^{WP}_g\left(\Psi_M\left(X;H\right) \right) \sim 2 \cdot H   \log X  .  
\]
\end{thm} 



Comparing Theorem~\ref{main thm}, and the Goldston--Montgomery prediction
$H\log \frac{X}{H}$ for primes in short intervals, there are two notable differences:

\bigskip

\begin{enumerate}[label=(\Alph*), ref=\Alph*, itemsep=1em]
\item \label{the factor of log X}
The factor of $\log X$, instead of $\log(X/H)$. 

\item \label{the factor 2}
The factor of $2$. 
\end{enumerate}

\bigskip

To explain \eqref{the factor of log X} and \eqref{the factor 2} we turn to the theory of higher degree L-functions. 

\subsection{Short interval variance for higher-degree L-functions}\label{sec:short int for L fns}

Let $L(s,\pi)$ be a primitive automorphic L-function associated
to a cuspidal automorphic representation of $\GL(d)$ over the
rationals (see \S\ref{sec:L-functions}). Its logarithmic derivative has the Dirichlet
series expansion
\[
-\frac{L'}{L}(s,\pi)
=
\sum_{n=1}^{\infty}\Lambda(n)a_\pi(n)n^{-s},
\qquad
\re(s)>1.
\]
We define the associated Chebyshev function
\[
\psi_\pi(x)
=
\sum_{n\le x}\Lambda(n)a_\pi(n),
\]
for which one has
\[
\psi_\pi(x)
=
O\!\left(
x\exp(-c\sqrt{\log x})
\right)
\]
for $d>1$, see \cite[Theorem 2.3]{LiuYe}.


Consider the variance
\[
\operatorname{Var}_\pi(X;H)
:=
\frac1X \int_1^X \bigl( \psi_\pi(x+H)-\psi_\pi(x)-m_\pi H \bigr)^2dx,
\]
where $m_\pi=1$   for the Riemann zeta function, and otherwise
$m_\pi=0$. 

  Bui, Keating and Smith~\cite{BKS} conjectured that for   degree $d>1$, there is a phase transition in the
variance of short interval sums at the scale $H=X^{1-1/d}$: for
``very short'' intervals,
\[
X^\varepsilon < H < X^{1-1/d-\varepsilon},
\]
one should have
\[
\operatorname{Var}_\pi(X;H)\sim H\log X,
\]
while for larger intervals,
\[
X^{1-1/d+\varepsilon}<H<X^{1-\varepsilon},
\]
the variance should instead satisfy
\[
\operatorname{Var}_\pi(X;H)\sim dH\log \frac{X}{H}.
\]

Their argument, following Goldston and Montgomery~\cite{GM},
relates these asymptotics to the form factor governing pair
correlation of zeros: 
let 
\[
F(X,T)=
\sum_{0\leq \gamma,\gamma'\leq T}
X^{i(\gamma-\gamma')}w(\gamma-\gamma'),
\qquad
w(u)=\frac{4}{4+u^2},
\]
where the sum runs over nontrivial zeros $\rho=\tfrac 12+i\gamma$ of $L(s,\pi)$.  
Writing
\[
X=T^{d\alpha},
\]
 we define the normalized form factor by
\[
K(\alpha)
:=
\lim_{T\to\infty}
\frac{F(T^{d\alpha},T)}{N(T)},
\]
assuming the limit exists. 
The pair correlation conjecture predicts that
\[
K(\alpha)=K_{\mathrm{GUE}}(\alpha)=\min(\alpha,1).
\]

 We show that in the very short interval regime, the relevant form
factor asymptotics already follow from the Riemann Hypothesis alone.
Thus, unlike the case of the Riemann zeta function ($d=1$),
RH alone already determines the early-time form factor. 
Consequently, the very short interval variance regime also follows from RH alone.
Our result is the following.

\begin{thm}\label{thm2}
Let $L(s,\pi)$ be a primitive automorphic $L$-function of degree
$d>1$, and assume RH for $L(s,\pi)$. Then the normalized form
factor satisfies
\[
K(\alpha)=\alpha,
\qquad
0<\alpha<\frac1d,
\]
recovering the GUE form factor in the early-time regime.
\end{thm}
%
 

This result helps clarify the appearance of the variance $H\log X$
in Theorem~\ref{main thm}. For $L$-functions of degree $d>1$, the very short interval regime
\[
H<X^{1-1/d}
\]
corresponds precisely to the early-time range
\[
\alpha<1/d,
\]
where the form factor is already determined by RH. By contrast,
for the Riemann zeta function ($d=1$), there is no corresponding very short interval regime,    
  and Goldston--Montgomery~\cite{GM} need the conjecture that the form factor $K(\alpha)$ coincides with the GUE form factor $K_{\GUE}(\alpha)= \min(\alpha,1)$ for $\alpha>1$.   For higher-degree
$L$-functions, however, the very short interval regime is already
determined by the early-time form factor. Since the spectrum of the
Laplacian on a hyperbolic surface has much higher density than the
zeros of any finite-degree $L$-function, random hyperbolic surfaces
may be viewed heuristically as an ``infinite-degree'' limit in
which only this early-time regime survives.
This explains  the appearance of the factor $\log X$ in
Theorem~\ref{main thm}.   

To explain   the  factor of $2$ in Theorem~\ref{main thm}, we replace  the GUE form factor by the GOE form factor expected for
generic hyperbolic surfaces.

Thus Theorem~\ref{main thm} may be viewed as the GOE and infinite-density
analogue of the early-time variance regime for higher-degree
$L$-functions.

\subsection{Lengths of closed geodesics and energy levels} 
It is worth noting that a related theme, that of correlations of lengths of closed geodesics on short scales,  has received attention in the physics literature in the context of understanding energy level statistics \cite{Argaman et al} with some  results on the mathematical side \cite{PS} for all surfaces, and special features for the arithmetic case \cite{LRSgeodsics}. Considerations related to this paper have been used to study the number variance  of the energy levels, when averaged over the moduli space \cite{RudGOE, RudWigmanCLT, RudWigmanAS, MarklofMonk}, and in those papers, as in this work, a crucial influence is the work of Mirzakhani and Petri \cite{MP}.

\subsection{Acknowledgements} 
Thanks to Yuxin He, Bingrong Huang, Jon Keating, Noam Pirani  and Yunhui Wu for their comments. 

  The author  would like to thank the Isaac Newton Institute for Mathematical Sciences, Cambridge, for support and hospitality during the programme Geometric spectral theory and applications, where work on this paper was undertaken. This work was supported by EPSRC grant EP/Z000580/1, and by the ISF-NSFC joint research program (Grant No. 3109/23).

\section{Background: Mirzakhani's integration formulas}

We quote  special cases that are relevant to us, combining Mirzakhani's integration formula (see \cite[Theorem 2.2]{MP}) with asymptotics of volume ratios \cite[Proposition 3.1]{MP} and  \cite[Lemma 22]{NieWuXue}. 

Let $S_g$ be a topological compact surface of genus $g$, and $\Mg$ the moduli space of hyperbolic metrics on $S_g$, equipped with the Weil-Petersson probability measure.


Give a reasonable function $F$ on $(0,\infty)$, we want to compute the expected value over the moduli space $\Mg$ of 
\[
F_{SNS}(M):=\sum_{\gamma\; SNS} F(\ell_M(\gamma))
\]  
where $M\in \Mg$, the sum is over all simple closed (unoriented) geodesics $\gamma$, which are non-separating, that is, the complement  $S_g\backslash \gamma\simeq S_{g-1,2}$ is a connected surface, of genus $g-1$ with two boundary components, and $\ell_M(\gamma)$ is the length of the geodesic. Then\footnote{For $g=2$ we get $1$ instead of $1/2$.} for $g>2$, 
\begin{equation}\label{eq:sum SNS}
\EWp(F_{SNS}) = \frac 12 \int_0^\infty F(\ell) \left( \frac{\sinh (\ell/2)}{\ell/2}\right)^2\left(1+O\left( \frac {\ell^2 }{g} \right) \right)  \,\ell \,d\ell .
\end{equation}

The second situation is for a bivariate function $F(x,y)$, we form the double sum
\[
F_{SNS}(M) = \sum_{(\gamma,\gamma') \; SNS} F\left( \ell_M(\gamma), \ell_M(\gamma') \right)
\]
where the sum is over pairs of simple non-separating geodesics, which are non-homotopic and disjoint: $\gamma\cap \gamma'=\emptyset$, and such that the complement $S_g\backslash \gamma\cup \gamma'\simeq S_{g-2,4}$ is still connected, so of genus $g-2$ with $4$ boundary components. Then 
\begin{multline}\label{eq:double sum SNS}
\EWp(F_{SNS}) =
\\
 \frac 1{2^2} \int_0^\infty\int_0^\infty  F(\ell,\ell') \left( \frac{\sinh (\ell/2)}{\ell/2}\right)^2 \left( \frac{\sinh (\ell'/2)}{\ell'/2}\right)^2 
\\
\left(1+O\left( \frac { \ell^2 +  \ell'^2}{g} \right) \right) \,\ell \,d\ell\,\ell'd\ell'  .
\end{multline}



 \section{The expected value}\label{sec:expectation}

\begin{prop}\label{prop:expected value Psi}
As $x\to \infty$,  $H=o(x)$ then 
$$\lim_{g\to \infty} \EWp(\Psi(x,x+H))\sim H.$$ 
\end{prop}
  \begin{proof}

We use the simple formula
\[
\Psi(x) = 2\sum_{\substack{\gamma\\ {\rm primitive}\\ {\rm unoriented}}} \sum_{\substack{k\geq 1\\ k\ell_\gamma\leq \log x}} \ell_\gamma = 
 2\sum_{\substack{\gamma\\ {\rm primitive}\\ {\rm unoriented}}}
\ell_\gamma \lfloor \frac{\log x}{\ell_\gamma} \rfloor.
\] 
Using  
\[
\lfloor \frac L \ell \rfloor = \sum_{1\leq n\leq L/\ell} 1 = \sum_{n\geq 1} 
\mathbf 1_{[0,\frac Ln]}(\ell)
\]
we find 
\begin{equation}\label{eq:formula for psi(x;H)}
\Psi(x;H) =  \sum_{\substack{\gamma\\ {\rm primitive}\\ {\rm unoriented}}}
f(\ell) 
\end{equation}
with
\begin{equation}\label{definition of f}
\begin{split}
f(\ell) &= 2 \ell \left( \left \lfloor \frac {B}{\ell} \right\rfloor -  \left \lfloor \frac {A}{\ell} \right\rfloor  \right)
=2\ell \, \sum_{n\geq 1} \mathbf 1_{( \frac An, \frac Bn]}(\ell)
\\
&A=\log x, \quad B=\log(x+H) .
\end{split}
\end{equation}


We write 
\[
\Psi(x;H)= \Psi_{SNS} +  \Psi_{\rest}
\]
 where $\Psi_{\rest}$ is the contribution of (primitive) geodesics which are not simple non-separating, that is either simple but separating, or non-simple.

We first bound $\EWp(\Psi_{\rest})$. 
Since each geodesic contributes 
\[
\begin{split}
f(\ell) =2\ell_\gamma \left( \lfloor \frac{\log (x+H)}{\ell_\gamma} \rfloor -  \lfloor \frac{\log x}{\ell_\gamma} \rfloor  \right)
&
\leq 2\ell_\gamma   \lfloor \frac{\log (x+H)}{\ell_\gamma} \rfloor 
\\
&\leq 2\log (x+H), 
\end{split}
\] 
we have an upper bound
\[
\Psi_{\rest} \leq 2\log(x+H)   N_{\rest}(\log x, \log (x+H))
\]
where $N_{\rest}(a,b)$ is the number of primitive, unoriented geodesics $\gamma$ which are not simple and non-separating, with $\ell_\gamma\in [a,b]$.   
Mirzakhani and Petri \cite{MP} showed that 
\[
\EWp(N_{\rest}(a,b)) \ll_b \frac 1g.
\]
This bound combines two cases: the contribution of simple separating geodesics is treated in the course of the proof of \cite[Proposition 4.2]{MP} while the non-simple geodesics are handled by  \cite[Proposition 4.5]{MP}. 
Hence we find
\begin{equation}
\EWp(\Psi_{\rest} ) \leq 2\log(x+H)  \EWp(N_{\rest}(\log x, \log(x+H)) \ll_x \frac 1g. 
\end{equation}

So it remains to evaluate the contribution $\Psi_{SNS}$ of simple, non-separating geodesics, for which we   use Mirzakhani's integration formula \eqref{eq:sum SNS}: 
\[
\begin{split}
\EWp \left( \Psi_{SNS}  \right)  &= 
\EWp(\sum_{\substack{\gamma\; SNS\\ {\mathrm{unoriented}}}} f(\ell(\gamma)) )
\\ 
&=\frac 12    \int_0^\infty f(\ell) \left( \frac {\sinh(\ell/2)}{\ell/2} \right)^2 \left( 1+O \left(\frac { \ell^2}g \right) \right)\, \ell \,d\ell 
\end{split}
\]
and hence
\begin{equation}\label{expectation SNS}
\lim_{g\to \infty} \EWp \left( \Psi_{SNS}(x;H) \right)   
  = \frac 12    \int_0^\infty f(\ell) \left( \frac {\sinh(\ell/2)}{\ell/2} \right)^2 \, \ell \,d\ell  .
\end{equation}

We compute the integral: 
\[
\begin{split} 
 \frac 12    \int_0^\infty f(\ell) \left( \frac {\sinh(\ell/2)}{\ell/2} \right)^2 \, \ell \,d\ell 
 & =   \int_{0}^{\infty}\left(  \sum_{n\geq 1} \mathbf 1_{[\frac An, \frac Bn]}(\ell) \right)   \left(2\sinh \frac \ell 2 \right)^2   d\ell 
\\
&= \sum_{n\geq 1} \int_{I(n)}  \left(2\sinh \frac \ell 2\right)^2  d\ell 
\end{split}
\]
where we denote 
$$
I(n):=[\frac An, \frac Bn].
$$  
The term $n=1$ gives  (recall $A=\log x$, $B=\log(x+H)$)
\begin{equation}\label{term n 1}
 \begin{split}
 \int_{A}^B  \left(2\sinh \frac \ell 2\right)^2  d\ell  &=x+H-x -2\log \frac{x+H}{x} + \frac 1x- \frac 1{x+H}\\
 &=H  + O\left(\frac{H }{x} \right)
 \end{split}
\end{equation}
which is the main term.  

Here, and in several places in the sequel, we use 
\begin{equation}\label{eq:log 1+H x}
\log \left( 1+\frac Hx \right) < \frac Hx
\end{equation}
which follows from the mean value theorem, which guarantees that there is some $\xi\in (0, \frac Hx)$ so that 
\[
  \log\left( 1+\frac Hx \right)   =\frac{H}{x}  \frac 1{1+\xi}<\frac Hx
 \]
as claimed. 

To bound the sum of the remaining terms, we use
\[
2\sinh\frac \ell 2\leq \begin{cases} 2\ell, &0\leq  \ell<1 \\ e^{\ell/2},&\ell>1 \end{cases}
\]
and observe that if $n\geq B$ then $I(n) = [\frac An,\frac Bn]\subset(0,1]$ and in that case we have
\[
\int_{I(n)}  \left(2\sinh\frac \ell 2\right)^2  d\ell \ll \int_{I(n)}  \ell^2 d\ell = \frac{B^3-A^3}{3n^3}.
\]
The contribution of the sum of such $n$'s is hence bounded by
\begin{multline}\label{eq:sum n grt B}
\sum_{n>B} \int_{I(n)} \left(2\sinh \frac \ell 2\right)^2  d\ell \ll (B^3-A^3)\sum_{n>B} \frac 1{n^3}
\\
 \ll (B-A)(B^2+AB+A^2)\frac 1{B^2} \ll B-A \ll \frac H{x} .
\end{multline}

For the sum over $2\leq n\leq B$, we use $2\sinh \frac \ell 2<e^{\ell/2}$ and hence
\begin{multline}\label{eq:sum n less B} 
\sum_{2\leq n\leq B} \int_{I(n)}  \left(2\sinh \frac \ell 2\right)^2  d\ell < \sum_{2\leq n\leq B} \int_{I(n)} e^\ell d\ell =\sum_{2\leq n\leq B} e^{B/n}-e^{A/n}
\\
=\sum_{2\leq n\leq B} x^{1/n}\left( (1+\frac Hx)^{1/n}-1 \right)
< x^{1/2} \frac{H}{2x}+x^{1/3}  \sum_{3\leq n\leq B}  \frac{H}{nx}\ll \frac{H}{x^{1/2}}.
\end{multline}

Putting together \eqref{term n 1}, \eqref{eq:sum n grt B} and \eqref{eq:sum n less B} gives
\[
\frac 12    \int_0^\infty f(\ell) \left( \frac {\sinh(\ell/2)}{\ell/2} \right)^2 \, \ell \,d\ell  = 
 H+ O\left(\frac{H}{x^{1/2}}\right). 
\]
 This concludes the proof of  Proposition~\ref{prop:expected value Psi}.  
\end{proof}

\section{The variance}

Our goal is to show that 
\begin{thm} \label{thm:variance psi}
As $x\to \infty$, $H=o(x)$,
\[
\lim_{g\to \infty} \Var_g\left( \Psi(x;H) \right) \sim 2 \cdot H \log x .
\] 
\end{thm}

\subsection{Plan of proof}

Recall 
\[
\psi(x;H) =  \sum_{\gamma} f(\ell_\gamma)
\]
where the sum is over  primitive unoriented geodesics, with 
\begin{equation*}
f(\ell) = 2\ell \sum_{n\geq 1} \mathbf 1_{[\frac An, \frac Bn]}(\ell)  
\quad 
A=\log x, \quad B=\log(x+H) .
\end{equation*}

We now compute the variance 
\[
\Var_g(\Psi ) = \EWp( \Psi^2 ) - (\EWp(\Psi))^2 .
\]
We expand $\Psi^2$ as a sum over pairs $(\gamma,\gamma')$ of  geodesics
\[
 \Psi^2 =  \sum_\gamma \sum_{\gamma'}  f(\ell(\gamma)) f(\ell(\gamma')).
\]
We separate out the diagonal pairs $\gamma=\gamma'$ from the rest 
\[
\EWp(\Psi) = \diag + \off 
\] 
and compute each term separately. We will show that  
$$\lim_{g\to \infty} \EWp(\off)=(\lim_{g\to \infty}\EWp(\Psi))^2$$ 
is exactly cancelled out by the squared expected value, so that 
\[
\lim_{g\to \infty} \Var_g(\Psi) = \lim_{g\to \infty} \EWp(\diag)
\]
and will then compute the diagonal term.

\subsection{The off-diagonal}
We decompose
\[
\off =\off_{SNS}  + \off_{\rest}
\]
where $\off_{SNS} $ is 
 the sum over off-diagonal SNS pairs, that is pairs $(\gamma,\gamma')$ of simple, non-intersecting  pairs so that $S_g\backslash \gamma\cup \gamma'$ is connected, so is a surface $S_{g-2,4}(\ell,\ell,\ell',\ell')$  of genus $g-2$ with two pairs of holes such that the  boundary lengths are the same in each pair; and $\off_{\rest}$ are the remaining off-diagonal pairs.

For $\off_{\rest} $ we bound 
\[
\off_{\rest}\leq \left( \max_{0<\ell\leq \log(x+H)} f(\ell) \right)^2     N'_2
\]
 where $N'_2$ is the number of off-diagonal, non SNS pairs\footnote{i.e. the two geodesics $\gamma$, $\gamma'$  are both simple but $S_g \backslash  \gamma \cup \gamma'$ is not connected, or at least one is not a simple non-separating geodesic, or the two intersect.} of  geodesics with length in $(0, \log(x+H)]$. Mirzakhani and Petri \cite[proof of Proposition 4.2 and Proposition 4.5]{MP} 
show that the expected value of $N_2'$ is $O_x(\frac 1g)$, so $\EWp\left(\off_{\rest}\right)\ll_x 1/g $ gives a negligible contribution as $g\to \infty$. 

To understand  $\EWp(\off_{SNS})$, we again use Mirzakhani's formula \eqref{eq:double sum SNS} with    $F(\ell,\ell') = f(\ell)f(\ell')$,   where $f(\ell)=2\ell (\lfloor \frac B \ell \rfloor-\lfloor \frac A \ell \rfloor)$ as in \eqref{definition of f}: 
\begin{multline*}
\EWp(  \off_{SNS}) = \EWp\left\{ \sum_{\substack {(\gamma,\gamma')\\ \gamma\cap \gamma'=\emptyset\\  \gamma\cup \gamma'\;{\rm non-separating}}} F(\ell(\gamma), \ell(\gamma'))   \right\} 
\\
=\frac 1{2^2} \int_0^\infty\int_0^\infty f(\ell)f(\ell')  \left( \frac{\sinh(\ell/2)}{\ell/2} \right)^2 \left( \frac{\sinh(\ell'/2)}{\ell'/2} \right)^2
\\
\left( 1+O\left( \frac{\ell^2+\ell'^2}{g} \right) \right) \; \ell \,d\ell \,\ell'\,d\ell',
\end{multline*}
so that 
\[
\lim_{g\to \infty} \EWp(  \off_{SNS})  =    \left( \frac 12\int_{0}^{\infty}    f(\ell) \left( \frac{\sinh(\ell/2)}{\ell/2} \right)^2 \,\ell \,d\ell \right)^2.	
\]
This  is just the square of the expression \eqref{expectation SNS} that we got for the expected value, and so is cancelled off by subtracting $\EWp(\Psi_{SNS})^2$.

\subsection{The diagonal term}
The diagonal term is 
\[
\diag  = \sum_\gamma f(\ell(\gamma))^2. 
\]
We further separate 
\[
\diag = \diag_{SNS} + \diag_{\rest}
\]
where $\diag_{SNS} $ is the sum over simple, non-separating geodesics, 
\[
\diag_{SNS} = \sum_{\gamma \, SNS} f( \ell(\gamma))^2 
\]
and $\off_{\rest} $ is the same sum over all the remaining geodesics. 

As in the computation of the expected value, we use  \cite{MP} to deduce that 
\[
\EWp(\diag_{\rest}) \leq 4 ( \log(x+H))^2\EWp( N_{\rest}(0, \log(x+H))) \ll_x\frac 1g. 
\]

Using Mirzakhani's formula \eqref{eq:sum SNS}, we have
\[
\EWp(\diag_{SNS}) =  \frac 12 \int_{0}^{\infty} f(\ell)^2   \left( \frac{\sinh \ell/2}{\ell/2} \right)^2\,\left( 1+O\left( \frac{\ell^2}{g} \right) \right)  \, \ell \, d\ell 
\]
so that 
\begin{equation} \label{sum over m and n}
\begin{split}
\lim_{g\to \infty}\EWp(\diag_{SNS}) 
&=  \frac 12 \int_{0}^{\infty} f(\ell)^2  \left( \frac{\sinh \ell/2}{\ell/2} \right)^2\,\ell\,d\ell 
\\
&=
2\int_{0}^{\infty} \left(\sum_{n\geq 1} \mathbf 1_{I(n)}(\ell)  \right)^2  \left(2\sinh \frac \ell 2 \right)^2\, \ell \, d\ell 
\\
&=2\sum_{ m,n\geq 1} \int_0^\infty 
\mathbf 1_{I(m)}(\ell) \mathbf 1_{I(n)}(\ell) \left(2\sinh \frac \ell 2 \right)^2 \, \ell \, d\ell 
\end{split}
\end{equation}
where we denoted   $I(n):=[\frac An, \frac Bn]$. 
Note 
\[
\mathbf 1_{I(m)} \mathbf 1_{I(n)} = \mathbf 1_{I(m)\cap I(n)}.
\] 

 The term $(m,n)=(1,1)$ gives 
\begin{multline*}
2\int_A^B   \left(2\sinh \frac \ell 2\right)^2 \, \ell \, d\ell   = 2\int_A^B (e^\ell-2+e^{-\ell}) \,\ell \, d\ell 
  =2H\log x + o(H)  
\end{multline*}
which is the main term.
We will  show that the sum of the terms $(m,n)\neq (1,1)$ gives $o(H)$.

\begin{lem}\label{lem:disjoint In} 
Assume  $0 < H < x$,  and set $A = \log x$, $B =\log(x+H)$, and consider the intervals $I(n):=[\frac An,\frac Bn]$.    
Assume that 
$$1 \le m <  \frac{x\log x}{H}, \quad n>m.
$$  
Then $I(m)$ and $I(n)$ are disjoint, and $I(n)$ lies to the left of $I(m)$ .
\end{lem}  
 \begin{proof}
We claim that  the smaller interval $I(n)$ lies entirely below the bigger interval $I(m)$, that is that  the right  endpoint $ B/n$ of the smaller interval $I(n)$  lies below the left endpoint $A/m$ of the bigger interval $I(m)$, 
meaning 
$$\frac Bn<\frac Am.
$$   
Since $m<n$, we have 
$\frac{n}{m}\ge 1+\frac{1}{m}$. 
Thus it suffices to show that
$$
\frac{B}{A}<1+\frac{1}{m},
$$
equivalently that 
$$
m<\frac{A}{B-A}= \frac{\log x}{ \log\left( 1+\frac Hx \right)} .
$$
Since we assume that $m<\frac{x\log x}{H}$, it suffices to show that 
\[
\frac {x\log x}H<\frac{\log x}{ \log\left( 1+\frac Hx \right)} \longleftrightarrow 
\log\left( 1+\frac Hx \right) <\frac Hx, 
\]
which  is \eqref{eq:log 1+H x}. 
\end{proof}

Among the pairs $(m,n)\neq (1,1)$, consider the diagonal terms $m=n\geq 2$.
Firstly, if $2\leq n\leq B$, then the intervals $I(n)$ are disjoint and all lie to the left of $I(2)$  by Lemma~\ref{lem:disjoint In},   and by monotonicity of $\left(2\sinh \frac \ell2 \right)^2$,  
\[
\begin{split}
2\sum_{2\leq n\leq B}  \int_{I(n)} \left(2\sinh \frac \ell2 \right)^2 \,\ell\, d\ell 
& < 
2B\int_{I(2)} \left(2\sinh \frac \ell2 \right)^2 \,\ell\, d\ell  
\\
&<2B\int_{A/2}^{B/2} e^\ell (\ell+1) d\ell 
\\
&= B\left( Be^{B/2}-Ae^{A/2}\right).
 \end{split}
\]
We have
\[
\begin{split}
B( Be^{B/2}-Ae^{A/2}) & = \log(x+H) \left(  (x+H)^{1/2}\log(x+H)-x^{1/2} \log x \right)
\\
& = x^{1/2}   \log(x+H)\left(  (1+\frac Hx)^{1/2}  \log(x+H) -\log x \right) 
\\
&\ll \frac{H (\log x)^2}{\sqrt{x}}  = o(H).
 \end{split}
\]

Next, we bound $m=n>B$: In that case, $I(n)\subset (0,1)$ and in that case we use $2\sinh \frac \ell 2\leq 2\ell$, $0<\ell<1$ so that 
\[
\int_{I(n)} \left( 2\sinh \frac \ell 2 \right)^2 \,\ell \,d \ell \ll \int_{A/n}^{B/n} \ell^3 d\ell = \frac{B^4-A^4}{4n^4}\leq \frac{(B-A) B^3}{n^4}, 
\] 
and the sum over $n>B$ is bounded by 
\[
\sum_{n>B} \int_{I(n)} \left( 2\sinh \frac \ell 2 \right)^2 \,\ell \,d \ell  \ll (B-A)  \cdot  B^3 \sum_{n>B} \frac 1{n^4} \ll 
B-A \ll \frac Hx.
\]
Thus the diagonal terms $m=n\geq 2$ contribute $o(H)$. 

It remains to bound the terms with $\frac{x\log x}H<m<n$.  
We write 
$$I(m)\cap I(n) = [\frac Am, \frac Bn]\subset (0, \frac Bn]$$
 for $m<n$, so that (we use $2\sinh \frac \ell 2\leq 2\ell$ here since $(0,\frac Bn]\subset (0,1]$ for $n>\frac{x\log x}{H}>B$)
\[
\int_{I(m)\cap I(n)} \left( 2\sinh \frac \ell 2 \right)^2 \,\ell d\ell \ll \int_0^{B/n} \ell^3 d\ell   <\frac{B^4}{ n^4}.
\] 

Summing over $n>m>x\log x/H$  gives 
\begin{multline*}
\sum_{n>m>\frac{x\log x}H} \int_{I(m)\cap I(n)} \left( 2\sinh \frac \ell 2 \right)^2 \,\ell d\ell 
\\
\ll
B^4 \sum_{n>\frac{x\log x}H} \frac 1{n^4} \sum_{\frac{x\log x}H<m<n} 1
 < B^4 \sum_{n>\frac{x\log x}H} \frac 1{n^3} 
\\
\ll B^4 \frac{H^2}{(x\log x)^2} \ll H \cdot \frac{ H(\log x)^2}{x^2} = o(H). 
\end{multline*}

Thus we find that if $x\to \infty$, $H=o(x)$, then
\begin{equation*}
\EWp(\diag_{SNS})
   = 2H \log x  +o(H)  .
\end{equation*}
 
 Altogether we obtain that as $x\to \infty$, $H=o(x)$, 
\[
\lim_{g\to \infty} \Var_g(\Psi_M) \sim 2 \; H \log x  
\]
proving Theorem~\ref{thm:variance psi}. \qed

\section{Comparison with the theory of L-functions}

\subsection{Background on L-functions}\label{sec:L-functions}

In order to explain the comparison between Theorem~\ref{main thm} and the conjecture of Goldston and Montgomery, we recall some of the theory of L-functions.

We shall deal with L-functions associated to cuspidal automorphic representations $\pi$ of $\GL(d)$ over the rationals   
 (when $d=1$ then these are either the Riemann zeta function or Dirichlet L-functions), see \cite{RSDuke}. The integer $d$ is called the degree of the L-function. 
They have an Euler product of the form 
\[
L(s,\pi) = \prod_p \prod_{j=1}^{d} \left(1-\alpha_\pi(p,j) p^{-s} \right)^{-1},\quad \re(s)>1 
\] 
where $d=d_\pi$ is the degree of the L-function. 
The Ramanujan-Petersson conjecture states that $|\alpha_\pi(p,j)|\leq 1$ (there are many known examples). 
These L-functions are primitive, i.e. cannot be written as a product of other L-functions.  

Such an L-function  admits an analytic continuation to all of the complex plane, save for a possible pole at $s=1$, which can only occur if the degree is $d=1$ and $L(s,\pi)$ is   the Riemann zeta function, since we assume that our L-functions are primitive. We denote by $m_\pi=0,1$ the order of the pole at $s=1$; $m_\pi=0$ unless $L(s,\pi)=\zeta(s)$. There is a functional equation of the form 
\[
\Phi(s,\pi):=Q_\pi^{s/2} \prod_{j=1}^d \Gamma_\R(s+\mu_F(j))L(s,F) = w \,\overline{\Phi(1-\bar s,\pi)},
\]  
with $Q_\pi>0$ a positive integer, $\Gamma_\R(s):=\pi^{-s/2}\Gamma\left( \frac s2 \right)$ and $|w|=1$.

The zeros $\{\rho_\pi\}$ of $\Phi(s,\pi)$ are called the nontrivial zeros of $L(s,\pi)$. 
The Riemann Hypothesis (GRH) asserts that all nontrivial zeros   lie on the line $\re(s)=1/2$.

The analogue of the Riemann von-Mangoldt theorem is 
\begin{equation}\label{Riemann-von Mangoldt}
N_\pi(T):=\#\{\rho_\pi: 0\leq \im(\rho_\pi)\leq T\} \sim d \frac{T \log T}{2\pi}, \quad T\to \infty. 
\end{equation}
In particular, the larger the degree, the greater the density of zeros.

The logarithmic derivative has an expansion
\[
-\frac{L'}{L}(s,\pi) = \sum_{n=1}^\infty \Lambda(n) a_\pi(n) n^{-s}, \quad \re(s)>1
\]
where for $p$ prime,
\[
a_\pi(p^k) =\sum_{j=1}^d \alpha_\pi(p,j)^k.  
\]

\subsection{The form factor at early times}
We now prove Theorem~\ref{thm2}.


%
  
We are given a primitive L-function $L(s,\pi)$ of degree $d>1$, associated to a cuspidal automorphic representation $\pi$ on $\GL(d)$ over the rationals. 
 We denote the nontrivial zeros of $L(s,\pi)$ by $\rho = \tfrac 12 +i\gamma$. 
Let 
\[
F(X,T) = \sum_{0\leq \gamma,\gamma'\leq T} X^{i(\gamma-\gamma')} w(\gamma-\gamma')
\]
with  $w(u) = \frac{4}{4+u^2}$. 
The normalized form factor is defined as
\[
K(\alpha):=\lim_{T\to \infty} \frac{F(T^{d\alpha},T)}{N(T)} 
\]
assuming the limit exists. 
\begin{thm}\label{thm:formfactor}
Let $\pi$ be a cuspidal automorphic representation of $\GL(d)$ over the rationals, with $d>1$. Assume that $L(s,\pi)$ satisfies RH. Then for $X=o(T)$, as $T\to \infty$,
\[
F(X,T) \sim   \frac 1{2\pi} T\log X . 
\]
Consequently, 
\[
K(\alpha) = \alpha, \quad \alpha<\frac 1d.
\]
\end{thm}
This recovers the GUE form factor $K_{\GUE}(\alpha)=\min(\alpha,1)$ in the  early-time regime $\alpha< 1/d$. 

Our proof of Theorem~\ref{thm:formfactor} follows that of Montgomery \cite{MontgomeryPC}, but we need to deal with the extra features of the theory of L-functions of higher degree and our ignorance of the Ramanujan-Petersson conjecture for most of these L-functions. We make use of Jiang's recent proof of Hypothesis H \cite{Jiang} to give a proof valid for all degrees, and all relevant L-functions.

\subsection{An explicit formula}
\begin{lem}\label{MVLemma}
If $1<\sigma<2$ and $x\geq 1$ then 
\begin{multline*}
(2\sigma-1)\sum_\gamma \frac{x^{i\gamma}}{(\sigma-\frac 12)^2 +(t-\gamma)^2} = 
\\
-x^{-1/2} \left( \sum_{n\leq x} \Lambda(n)a(n) \left( \frac{x}{n} \right)^{1-\sigma+it} + \sum_{n>x} \Lambda(n) a(n) \left( \frac xn\right)^{\sigma+it} \right) 
\\
+O\left(x^{1/2-\sigma}\log(|t|+2)+\frac{x^{1/2}}{|t|+2} \right) . 
\end{multline*}
\end{lem}
The proof of Lemma~\ref{MVLemma} is identical to that of the corresponding Lemma   in \cite{MontgomeryPC}    (we do not need RH for this if we don't insist that $\gamma$ are real \cite[Proposition 1]{Goldston notes}).

We use Lemma~\ref{MVLemma} with $\sigma=\frac 32$, and express the Lemma as the statement $L(x,t) = R(x,t)$. 

\subsection{The LHS}
As in \cite{MontgomeryPC}, we have 
\[
\int_0^T |L(x,t)|^2 dt  = 2\pi  \sum_{0\leq \gamma,\gamma'\leq T} x^{i(\gamma-\gamma')} w(\gamma-\gamma') + O\left( \left( \log T \right)^3 \right) 
\]
with $w(u) = 4/(4+u^2)$. 
Thus as $T\to \infty$, 
\begin{equation}\label{eq:LHS}
\int_0^T |L(x,t)|^2 dt  \sim  2\pi F(x,T).
\end{equation}

\subsection{The RHS}

\begin{lem}
 \[
\int_0^T |R(x,t)|^2 dt = T \log x +O(x\log x). 
\]
\end{lem}
\begin{proof}
Ignoring the remainder term, whose effect is treated as in \cite{MontgomeryPC}, we focus on 
\begin{multline*}
x^{-1} \int_0^T \left|  \sum_{n\leq x} \Lambda(n)a(n) \left( \frac{x}{n} \right)^{-\frac 12+it} + \sum_{n>x} \Lambda(n) a(n) \left( \frac xn\right)^{\frac 32+it} \right|^2 dt
\\
=\int_0^T \left| \sum_n b(n) n^{-it} \right|^2 dt
\end{multline*}
where we take 
\begin{equation}\label{def of b(n)}
b(n) = \begin{cases} x^{-1} \Lambda(n)a(n)n^{\frac 12}, & n\leq x \\
\\
x  \Lambda(n)a(n) n^{-\frac 32 }, & n>x .
\end{cases}  
\end{equation} 

We use the Montgomery-Vaughan mean value theorem (approximate orthogonality of $n^{-it}$)  \cite[Corollary 3]{MVHilbert}: If $\sum_n  n\cdot |b(n)|^2   <\infty$  then 
\[
\int_0^T \left| \sum_n b(n) n^{-it} \right|^2 dt = \sum_n |b(n)|^2 \left( T+O\left( n \right) \right). 
\]

\begin{lem}
Let $b(n)$ be given by \eqref{def of b(n)}. Then 
\[
\sum_{n\geq 1} |b(n)|^2 \sim  \log x
\]
and 
\[
\sum_{n\geq 1}  n |b(n)|^2 \ll x\log x. 
\]
\end{lem}
\begin{proof}
 Recall the PNT for automorphic L-functions: If $\pi$ is cuspidal automorphic    then \cite[Lemma 5.1]{LiuWangYe}
\[
\sum_{n \leq x} \Lambda(n)|a_\pi(n)|^2 \sim x.
\]
Applying summation by parts gives
\begin{equation}\label{bound for a(n)2 Lambda(n) log n}
\sum_{n \leq x} \log(n) \, \Lambda(n)|a_\pi(n)|^2 \sim x \log x.
\end{equation}
The difference between \eqref{bound for a(n)2 Lambda(n) log n} and $\sum_{n \leq x} \Lambda(n)^2 |a_\pi(n)|^2$ is
\begin{multline}\label{eq:difference between two sums} 
0\leq \sum_{n \leq x} \log(n) \, \Lambda(n)|a_\pi(n)|^2 -\sum_{n \leq x} \Lambda(n)^2 |a_\pi(n)|^2 
\\
= \sum_{n \leq x} (\log(n)-\Lambda(n))  \, \Lambda(n)|a_\pi(n)|^2
\leq \log x \cdot   \sum_{\substack{ p^k\leq x\\k\geq 2}}  \log p |a_\pi(p^k)|^2 .
\end{multline}
We have\footnote{When $p$ is unramified for $\pi$ this is an equality, for the finitely many ramified primes see \cite[Lemma 2.2]{ST}.} 
$$
|a_\pi (p^k)|^2\leq a_{\pi \times \check{\pi}}(p^k)
$$ 
where $a_{\pi \times \check{\pi}}$ appears in  the  logarithmic derivative of the Rankin-Selberg L-function as
\[
-\frac{L'}{L}(s,\pi \times \check{\pi}) = \sum_n \Lambda(n)a_{\pi \times \check{\pi}}(n)n^{-s}.
\]
Hence the RHS of \eqref{eq:difference between two sums}  is bounded by $\Delta \cdot \log x$, where
\[
\Delta:=\sum_{\substack{p^k\leq x\\k\geq 2}} a_{\pi \times \check{\pi}}(p^k) \log p.
\]
In the proof of \cite[Theorem 8.1]{Jiang}, we find the bound 
\[
\Delta \ll_{\pi,\varepsilon} x^{1-\frac 1{d^2+1}+\varepsilon},
\] 
for all $\varepsilon>0$, obtained using Rankin's trick. Hence we deduce
\[
\sum_{n \leq x} \log(n) \, \Lambda(n)|a_\pi(n)|^2 -\sum_{n \leq x} \Lambda(n)^2 |a_\pi(n)|^2 =o(x).
\]
Using \eqref{bound for a(n)2 Lambda(n) log n} we obtain
\begin{equation}\label{eq:LWY and Jiang}
\sum_{n \leq x} \Lambda(n)^2 |a_\pi(n)|^2 \sim x\log x  .
\end{equation}

Therefore, using Abel summation, 
\[
\sum_{n \leq x} \Lambda(n)^2 |a_\pi(n)|^2 n \sim \frac 12x^2\log x   
\]
and 
\[
 \sum_{n>x} \frac{\Lambda(n)^2 |a(n)|^2}{n^3} \sim \frac 12 \frac{\log x}{ x^2}.
\]
and hence we have
\[
\sum_{n\leq x} |b(n)|^2 = x^{-2}\sum_{n\leq x} \Lambda(n)^2 |a(n)|^2 n \sim \frac 12  \log x 
\]
and
\[
\sum_{n>x} |b(n)|^2  = x^2 \sum_{n>x} \frac{\Lambda(n)^2 |a(n)|^2}{n^3} \sim \frac 12 \log x
\]
so that
\[
\sum_{n\geq 1} |b(n)|^2 \sim  \log x . 
\]

Further, again using Abel summation from \eqref{eq:LWY and Jiang}, 
\[
\sum_{n\leq x} |b(n)|^2 n = x^{-2}\sum_{n\leq x} \Lambda(n)^2 |a(n)|^2 n^2 \sim \frac 13 x  \log x
\]
and 
\[
\sum_{n>  x} |b(n)|^2 n =x^2 \sum_{n>x} \frac{\Lambda(n)^2 |a(n)|^2}{n^2} \sim  x\log x
\]
so that 
\[
\sum_{n\geq 1} |b(n)|^2 n \ll x \log x .
\]
\end{proof}

We deduce
\[
\begin{split}
 \int_0^T |R(x,t)|^2=  \int_0^T \left|\sum_{n} b(n)n^{-i t} \right|^2 dt &=   \sum_n |b(n)|^2 (T+O(n))
\\
&=   T\log x + O\left(  x\log x\right) .
\end{split}
\]
Inserting into \eqref{eq:LHS} proves the claim. 
 \end{proof}

\section{Explaining the factor of $2$}\label{sec:GOE}

We now explain heuristically the origin of the factor of $2$ in
Theorem~\ref{main thm}. The point is that the variance of short interval sums 
 is determined by pair correlation of the relevant spectrum through the form factor.
 For zeros of L-functions one expects GUE statistics,
while for generic hyperbolic surfaces one expects GOE statistics.
The difference between these two symmetry classes produces the
factor of $2$.

We briefly recall the relevant picture. In the work of Goldston and
Montgomery~\cite{GM}, and its extension to higher-degree
L-functions by Bui, Keating and Smith~\cite{BKS}, the variance of
short interval sums is related, via the explicit formula, to the form
factor $K(\alpha)$, 
which governs pair correlation of zeros. In the classical setting of
zeros of L-functions, one expects GUE statistics, corresponding to
the GUE form factor
\[
K_{\mathrm{GUE}}(\alpha)=\min(\alpha,1).
\]

For generic hyperbolic surfaces, however,  some 
evidence suggests GOE statistics for the Laplace eigenvalues \cite{BGS, AS, RudGOE}. 
The corresponding GOE form factor satisfies
\[
K_{\mathrm{GOE}}(\alpha)\sim 2\alpha,
\qquad
\alpha\to 0,
\]
whereas
\[
K_{\mathrm{GUE}}(\alpha)\sim \alpha,
\qquad
\alpha\to 0.
\]
Thus the GOE form factor is asymptotically twice as large as the GUE
form factor at early times.

For higher-degree L-functions, the relevant regime for very short
intervals is precisely the early-time range
\[
\alpha<\frac1d.
\]
As explained in \S\ref{sec:short int for L fns}, random hyperbolic surfaces should be viewed
heuristically as an ``infinite-degree'' limit, since the spectrum of
the Laplacian has much greater density than the zeros of any
finite-degree L-function. Consequently, only the early-time
behavior of the form factor contributes to the variance.

In the Goldston--Montgomery argument,
the variance is determined by the form factor through an integral
formula. Consequently, replacing the GUE form factor by the GOE
form factor multiplies the resulting variance by $2$. Thus the expected GOE statistics for the
Laplace spectrum of generic hyperbolic surfaces lead naturally to
expect 
\[
\operatorname{Var}(\Psi_M(X;H))
\sim
2H\log X.
\]

Accordingly, Theorem~\ref{main thm} may be viewed as the GOE analogue of the
early-time variance regime for higher-degree L-functions.

\end{document}